\documentclass[twoside,12pt]{article}
\usepackage{amsmath,amssymb,amsthm,amsfonts}
\usepackage{bezier}
\newcounter{mycounter}
\newtheorem{definition}{Definition}[section]
\newtheorem{theorem}[definition]{Theorem}

\begin{document}
\begin{center}
\ \\
\ \\
\ \\
{\bf Weak solutions for a system of quasilinear elliptic equations}\ \\
\ \\
M. A. Ragusa$^a$ and A. Razani$^b$ \\ \ \\
$^a$Department of Mathematics, University of Catania, Viale
Andrea Doria No. 6, Catania 95128, Italy and RUDN University, Miklukho-Maklaya Str. 6, Moscow, 117198, Russia.
E-Mail: maragusa@dmi.unict.it\\
$^b$Department of Pure Mathematics, Faculty of Science, Imam Khomeini International University,
Postal code: 34149-16818, Qazvin, Iran. E-Mail: razani@sci.ikiu.ac.ir \\
\ \\
\ \\
\begin{minipage}[pos]{12cm}
\indent A system of quasilinear elliptic equations on an unbounded domain is considered.
The existence of a sequence of radially symmetric weak solutions is proved via variational methods.
\end{minipage}
\end{center}
{\it \bf $2010$ Mathematics subject classification:} 35J20, 34B10, 35J50.\\
\ \\
\section{Introduction}
We consider the following problem
\begin{equation}\label{eq1.1}
\begin{cases}
-\Delta_p u+|u|^{p-2}u=\lambda \alpha_1(x) f_1(v) & \text{in $\mathbb{R}^N$},\\
 -\Delta_q v+|v|^{q-2}v=\lambda \alpha_2(x) f_2(u) & \text{in $\mathbb{R}^N$},\\
 u,v\in W^{1,p}(\mathbb{R}^N),
\end{cases}
\end{equation}
where $p,q> N > 1$. We assume that $f_1,f_2:\mathbb{R}\to \mathbb{R}$ are continuous functions,
$\alpha_1(x), \alpha_2(x)\in L^1(\mathbb{R}^N)\cap L^{\infty}(\mathbb{R}^N)$ are nonnegative (not identically zero) radially symmetric maps,
and $\lambda $ is a  real parameter.
Also $\Delta_pu:=\mathrm{div}(|\nabla u|^{p-2}\nabla u)$ denotes the $p$-Laplacian operator.

Partial differential equations's  is used to model a wide variety
of physically significant problems arising in every different areas such as physics, engineering and other
applied disciplines (see \cite{Ehsani,hesaraki01,hesaraki02,Lindgren,Mokhtar,Mokhtarp,Polidoro2008, pournaki,pournakir,RaggusaDMJ,Ragusa2020,Razani2002,Razani2002JMAA,Razani2004JMAA, Razani2004BAMS,Razani2007,Razani2012,Razani2014, Razani2018,Razani2019,razanigo,ScapllatoBVP}). Sobolev spaces play an important role in the theory of partial differential equations as well as
Orlicz-Morrey space and  $\dot{B}_{\infty,\infty}^{-1}$ space (see \cite{BGR,GalaApplicable,GalaAIMS, Gala2016, R2001, RHERZ,Ragusa2020}). Laplace equation is the prototype for linear elliptic equations. This equation has a non-linear counterpart, the so-called $p$-Laplace
equation (see \cite{BehboudiFilomat,Chen, KhalkhaliInfinitely, Khalkhali2013, MahdaviFilomat,MakvandTJM,MakvandCKMS,MakvandFilomat,MakvandGMJ,Yang}).

Here, by inspiration of \cite{MakvandMJM}, we prove the existence of a sequence of radially symmetric weak solutions for
\eqref{eq1.1} in the unbounded domain $\mathbb{R}^N$.

The solution of \eqref{eq1.1} belongs to the product space
\[
W^{1, (p,q)}(\mathbb{R}^N)=W^{1,p}(\mathbb{R}^N)\times W^{1,q}(\mathbb{R}^N)
\]
equipped with the norm
$\|(u,v)\|_{(p,q)}=\|u\|_p+\|u\|_q$.
\begin{definition}
For fixed $\lambda_1$ and $\lambda_2$,  $((u,v):\mathbb{R^N}\rightarrow\mathbb{R}$ is said to be a weak solution of \eqref{eq1.1},
if $(u,,v) \in W^{1,(p,q)}(\mathbb{R}^N)$ and for every $(z,w)\in W^{1,(p,q)}(\mathbb{R}^N)$
\[
\begin{array}{rl}
&-(\int_{\mathbb{R}^N}|\nabla
u(x)|^{p-2}\nabla u(x).\nabla z(x)dx-(\int_{\mathbb{R}^N}|\nabla
u(x)|^{q-2}\nabla u(x).\nabla w(x)dx\\ & \\
&+\int_\mathbb{R^N}|u(x)|^{p-2}u(x)z(x)dx+\int_\mathbb{R^N}|v(x)|^{q-2}v(x)w(x)dx\\ & \\
&\qquad -\lambda_1\int_\mathbb{R^N}\alpha_1(x)f_1(v(x))z(x)dx-\lambda_2\int_\mathbb{R^N}\alpha_2(x)f_2(u(x))w(x)dx=0,
\end{array}
\]
where
\[
\begin{array}{rl}
\|(u,v)\|_{W^{1,(p,q)}(\mathbb{R}^N)}:=&\left(\int_\mathbb{R^N} |\nabla u(x)|^p dx + \int_\mathbb{R^N}|u(x)|^pdx\right)\\
& \qquad +\left(\int_\mathbb{R^N} |\nabla v(x)|^q dx + \int_\mathbb{R^N}|v(x)|^qdx\right).
\end{array}
\]
\end{definition}

Later no, note that the critical points of a energy functional are exactly the weak solutions of \eqref{eq1.1}.

Morrey's theorem, implies the continuous embedding
\begin{equation}\label{embedding}
W^{1,(p,q)}(\mathbb{R^N})\hookrightarrow L^{\infty}(\mathbb{R^N})\times L^{\infty}(\mathbb{R^N}),
\end{equation}
which says that there exists $c$ (depends on $p,q,N$), such that
$\|(u,v)\|_{\infty}\leq c\|(u,v)\|_{W^{1,(p,q)}(\mathbb{R}^N)}$, for every $(u,v) \in W^{1,p}(\mathbb{R^N})\times W^{1,q}(\mathbb{R^N})$, where $\|(u,v)\|_\infty:=\max \{\|u\|_\infty,\|v\|_\infty\}$. Since in the low-dimensional case, every function $(u,v)\in W^{1,(p,q)}(\mathbb{R}^N)$
admits a continuous representation (see \cite[p.166]{Brezis}). In the sequel we will replace $(u,v)$ by this element.

We need the following notations (see \cite{candito} or \cite{Li} for more details):
\begin{itemize}
\item[(I)] $O(N)$ stands for the orthogonal group of $R^N$.
\item[(II)] $B(0,s)$ denotes the open $N$-dimensional ball of center zero, radius $s>0$ and standard Lebesgue measure, $meas(B(0,s))$.
\item[(III)] $\|\alpha\|_{B(0,\frac{s}{2})}:=\int_{B(0,\frac{s}{2})}\alpha(x) dx$.
\end{itemize}

\begin{definition}\

\begin{itemize}
\item A function $h : \mathbb{R}^N\to \mathbb{R}$ is radially symmetric if $h(gx) = h(x)$, for every
$g\in O(N)$ and $x\in \mathbb{R}^N$.
\item
Let $G$ be a topological group. A continuous map $\xi:G\times X\to X:(g,x)\to \xi(g,u):=gu$, is called the action of $G$ on the Banach space $(X,\|.\|_X)$ if
\[
1u=u,\qquad (gm)u=g(mu),\qquad u\mapsto gu\ \text{is linear}.
\]
\item  The action is said to be isometric if $\|gu\|_X=\|u\|_X$, for every $g\in G$.
\item  The space of $G$-invariant points is defined by
\[
Fix(G):=\{ u\in X: gu=u, \text{for all}\ {g\in G}\}.
\]
\item A map $m:X\to R$ is said to be $G$-invariant if $mog=m$ for every $g\in G$.
\end{itemize}
\end{definition}
The following theorem is important to study the critical point of the functional.
\begin{theorem}\label{palais}
(Palais (1979)) Assume that the action of the topological group $G$ on the Banach space $X$ is isometric.
If $J\in C^1(X:\mathbb{R})$ is $G$-invariant
and if $u$ is a critical point of $J$ restricted to $Fix(G)$, then $u$ is a critical point of $J$.
\end{theorem}

The action of the group $O(N)$ on $W^{1,p}(R^N)$ can be defined by $(gu)(x):=u(g^{-1}x)$, for every
$g\in W^{1,p}(\mathbb{R}^N)$ and $x\in \mathbb{R^N}$.
A computation shows that this group acts linearly and isometrically, which means $\|u\|=\|gu\|$, for every
$g\in O(N)$ and $u\in W^{1,p}(\mathbb{R}^N)$.

\begin{definition}
The subspace of radially symmetric functions of ${W_r^{1,(p,q)}(\mathbb{R}^N)}$ is defined by
\[
\begin{array}{rl}
X:=&W_r^{1,(p,q)}(\mathbb{R}^N)\\
:=& \{ (u,v)\in W^{1,(p,q)}(\mathbb{R}^N): (g_1u,g_2v)=(u,v),\ \text{for all}\ (g_1,g_2)\in O(N)\times O(N)\},
\end{array}
\]
and endowed by the norm

\[
\begin{array}{rl}
\|(u,v)\|_{W_r^{1,(p,q)}(\mathbb{R}^N)}:=&\left(\int_\mathbb{R^N} |\nabla u(x)|^p dx + \int_\mathbb{R^N}|u(x)|^pdx\right)\\
& \qquad +\left(\int_\mathbb{R^N} |\nabla v(x)|^q dx + \int_\mathbb{R^N}|v(x)|^qdx\right).
\end{array}
\]
\end{definition}
In what follows: $\|(u,v)\|_{r}$ denotes $\|(u,v)\|_{W_r^{1,(p,q)}(\mathbb{R}^N)}$.

The following crucial embedding result due to Krist\'{a}ly and principally based on a Strauss-type estimation (see \cite{Strauss})
(Also see \cite[Theorem 3.1]{Kristaly}, \cite{Varga} and \cite{Willem} for related subjects).
\begin{theorem}\label{thcom}
The embedding $W^{1,p}_r(\mathbb{R}^N)\hookrightarrow L^{\infty}(\mathbb{R}^N)$, is compact whenever $2\leq N<p<+\infty$.
\end{theorem}

Here we consider the following functionals:
\begin{itemize}
\item $F_i(\xi):=\int_{0}^{\xi}f_i(t)dt$ for every $\xi\in \mathbb{R}$.
\item $\Phi(u,v):=\frac{\|u\|_r^p}{p}+\frac{\|v\|_r^q}{q} $ for every $(u,v) \in X$.
\item $\Psi(u,v):=\int_\mathbb{R^N}\alpha_1F_1(v(x))dx+\int_\mathbb{R^N}\alpha_2F_2(u(x))dx$, for every $(u,v) \in X$.
\item $I_\lambda(u,v):=\Phi(u,v)-\lambda \Psi(u,v)$ for every $(u,v) \in X$.
\end{itemize}

By standard arguments \cite{candito}, we can show that $\Phi$ is G\^ateaux differentiable, coercive and
sequentially weakly lower semicontinuous whose derivative at the point $(u,v) \in X$ is the functional $\Phi'(u,v)\in X^*$ given by
\[
\begin{array}{rl}
\Phi'(u,v)(z,w)=&(\int_\mathbb{R^N}|\nabla u|^{p-2}\nabla u.\nabla z dx + \int_\mathbb{R^N}|u|^{p-2}uz dx)\\
+&(\int_\mathbb{R^N}|\nabla v|^{q-2}\nabla v.\nabla w dx + \int_\mathbb{R^N}|v|^{q-2}vw dx),
\end{array}
\]
for every $(z,w)\in X$. Also standard arguments show that the functional $\Psi_i$ are well defined,
sequentially weakly upper semicontinuous and G\^ateaux differentiable whose G\^ateaux derivative
at the point $(u,v) \in X$ and for every $(z,w)\in X$ is given by,
\[
\Psi'(u,v)(z, w)=\int_\mathbb{R^N}\alpha_1(x)f_1(u(x)) dx+\int_\mathbb{R^N}\alpha_2(x)f_2(v(x)) dx.
\]

\section{Weak solutions}
First we recall  the following theorem \cite[Theorem 2.1]{bonanno}.
\begin{theorem}\label{theo2.1}
Let $X$ be a reflexive real Banach space, let $\Phi,\Psi:X\to \mathrm{R}$ be two G\^ateaux differentiable
functionals such that $\Phi$ is sequentially weakly lower semicontinuous, strongly continuous and coercive,
and $\Psi$ is sequentially weakly upper semicontinuous. For every $r>\inf_X\Phi$, set
\[
\begin{aligned}
\varphi(r)&:=\inf_{\Phi(u)<r}\frac{\sup_{\Phi(v)<r}\Psi(v)-\Psi(u)}{r-\Phi(u)},\\
\gamma&:=\liminf_{r\to+\infty}\varphi(r),\quad \text{and}\quad \delta:=\liminf_{r\to(\inf_X\Phi)^+}\varphi(r).
\end{aligned}
\]
Then the following properties hold:
\begin{itemize}
\item[($a$)] for every $r>\inf_X\Phi$ and every $\lambda \in]0,\frac{1}{\varphi(r)}[$, the restriction of the functional
$ I_\lambda:=\Phi-\lambda \Psi $
to $\Phi^{-1}(]-\infty,r[)$ admits a global minimum, which is a critical point (local minimum) of $I_\lambda$ in $X$.
\item[($b$)] if $\gamma<+\infty$, then for each $\lambda \in]0,\frac{1}{\gamma}[$, the following alternative holds either,
\begin{itemize}
\item[($b_1$)] $I_\lambda$ possesses a global minimum, or
\item[($b_2$)] there is a sequence $\{u_n\}$ of critical points (local minima) of $I_\lambda$ such that
$ \lim_{n\to+\infty}\Phi(u_n)=+\infty$.
\end{itemize}
\item[($c$)] if $\delta<+\infty$, then for each $\lambda \in]0,\frac{1}{\delta}[$, the following alternative holds either:
\begin{itemize}
\item[($c_1$)] there is a global minimum of $\Phi$ which is a local minimum of $I_\lambda$. or,
\item[($c_2$)] there is a sequence $\{u_n\}$ of pairwise distinct critical points (local minima) of $I_\lambda$ which weakly converges to a global minimum of $\Phi$, with $\lim_{n\to+\infty}\Phi(u_n)=\inf_{u\in X}\Phi(u)$.
\end{itemize}
\end{itemize}
\end{theorem}
For fixed $D>0$, set
\[
m(D):=meas(B(0,D))=D^N\frac{\pi^{\frac{N}{2}}}{\Gamma(1+\frac{N}{2})},
\]
where $\Gamma$ is the Gamma function defined by $\Gamma(t):=\int_0^{+\infty}z^{t-1}e^{-z}dz$ for all $t>0$.  Moreover,
\begin{equation}\label{k}
 \Omega:=\max \left\{\frac{m(D)\left(\frac{\sigma(N,p)}{D^p}+g(p,N)\right)}{p\lambda B_1\|\alpha_2\|_{B(0,\frac{D}{2})}},\frac{m(D)\left(\frac{\sigma(N,q)}{D^q}+g(q,N)\right)}{q\lambda B_2\|\alpha_1\|_{B(0,\frac{D}{2})}}\right\}>0,
\end{equation}
where $\sigma(N,p):=2^{p-N}(2^N-1)$, $c= \frac{2p}{2-N}$, $m_1,m_0$ are upper and lower bounds for $M(t)$ in \eqref{eq1.1} and
\[g(p,N):=\frac{1+2^{N+p}N\int_{\frac{1}{2}}^1
t^{N-1}(1-t)^pdt}{2^N}.\]
Assume $\|\cdot\|_1$ denotes the norm of $L^1(\Omega)$ and $F(\xi):=F_1(\xi)+F_2(\xi)$.
\begin{theorem}\label{theo1}
Let $f_i:\mathbb{R}\rightarrow\mathbb{R}$ be two continuous and radially symmetric functions. Set
\[
\begin{array}{rl}
A:=&\underset{(\xi_1,\xi_2)\to +\infty}{\liminf}\frac{\underset{|t_1|\leq \xi_1}{\max}F_2(t_1)}{|\xi_1|^p}+
\frac{\underset{|t_2|\leq\xi_2}{\max}F_1(t_2)}{|\xi_2|^q},\\
B_1:=&\underset{\xi_2\to +\infty}{\limsup}\frac{F_1(\xi_2)}{|\xi_2|^p},\ \text{and}\ B_2:=\underset{\xi_1\to +\infty}{\limsup}\frac{F_2(\xi_1)}{|\xi_1|^q}.
\end{array}
\]
where $B:=B_1+B_2$, $\xi=(\xi_1,\xi_2)$.
If $\underset{(\xi_1,\xi_2)\geq 0}{\inf}F_2(\xi_1)+F_1(\xi_2)=0$ and $A<\Omega m_0 B$, where $\Omega$ is given by \eqref{k},
for every $\lambda\in \Lambda:=\left]\Omega,\frac{1}{\left({pc_1^p}\|\alpha_2\|_1
	+{qc_2^q}\|\alpha_1\|_1\right)A}\right[$,
there exists an unbounded sequence of radially symmetric weak solutions for \eqref{eq1.1} in $X$.
\end{theorem}
\begin{proof}
For fixed $\lambda \in \Lambda$, we consider $\Phi$, $\Psi$ and $I_\lambda$ as in the last section. Knowing that $\Phi$ and
$\Psi$ satisfy the regularity assumptions in Theorem \ref{theo2.1}. In order to study the critical points of
$I_\lambda$ in $X$, we show that $\lambda <\frac{1}{\gamma}<+\infty$, where $\gamma=\underset{r\to +\infty}{\liminf}\phi(r)$. Let $\lbrace t_n \rbrace $ be a sequence of positive numbers such that $\lim_{n\to \infty} t_n=+\infty $, $r_{1n}:=\frac{t_{1n}^p}{pc_1^p}$ and $r_{2n}:=\frac{t_{2n}^q}{qc_2^q}$, for all $n \in \mathbb{N}$. Set $r_n=\min\{r_{n1},r_{n2}\}$. Considering Theorem \ref{thcom} (by relation \eqref{embedding}), a computation shows that
\begin{equation}\label{phii}
\begin{aligned}
\Phi^{-1}(]-\infty,r_{n}[)=&\{(z,w)\in X: \Phi(z,w)<r_{n}\}\\
=& \{(z,w)\in X: \frac{\|z\|_r^p}{p}+\frac{\|w\|_r^q}{q}<r_n\}
\\
\subset& \{(z,w)\in X;\, \|(z,w)\|_{\infty}<t_n\},
\end{aligned}
\end{equation}
where $t_n=\min\{t_{n1},t_{n2}\}$.

Since $\Phi(0,0)=\Psi(0,0)=0$, by a computation one can show
\[
\begin{aligned}
\varphi(r_n)&=\inf_{\Phi(u,v)<r_n}\frac{(\sup_{\Phi(z,w)<r_n}\Psi(z,w))-\Psi(u,v)}{r_n-\Phi(u,v)}\\
&\leq \left({pc_1^p}\|\alpha_2\|_1
+{qc_2^q}\|\alpha_1\|_1\right)A.
\end{aligned}
\]
Hence
\[
\gamma\leq \liminf_{n\to +\infty}\varphi(r_n)\leq \left({pc_1^p}\|\alpha_2\|_1
+{qc_2^q}\|\alpha_1\|_1\right)A<+\infty.
\]
Now, we show that $I_\lambda$ is unbounded from below. Let $\{d_{1n}\}$ and $\{d_{2n}\}$ be two sequences of positive numbers such that $\lim_{n \rightarrow +\infty}d_{1n}=\lim_{n \rightarrow +\infty}d_{2n}=+\infty$ and
\begin{equation}\label{B}
B_1=\lim_{n\rightarrow +\infty}\frac{F_1(d_{2n})}{d_{2n}^q},\ B_2=\lim_{n\rightarrow +\infty}\frac{F_2(d_{1n})}{d_{1n}^p}
\end{equation}
Define $\{(H_{1n}, H_{2n})\} \in X $ by
\begin{equation*}
 H_{in}(x):=\left\{\begin{array}{lll}
\displaystyle0  & \quad \mathbb{R^N}\setminus B(0,D)\\
\displaystyle d_{in}  &  \quad B(0,\frac{D}{2})
\\
\displaystyle \frac{2d_{in}}{D}(D-|x|)  & \quad B(0,D)\setminus
B(0,\frac{D}{2}),
\end{array}\right.
\end{equation*}
for every $n \in \mathbb{N}$ and $i=1,2$. By a similar argument and computations in \cite[P.1017]{candito} one can show that
\[
\begin{array}{rl}
\|H_{2n}\|_r^q=&d_{2n}^qm(D)\left(\frac{\sigma(N,p)}{D^q}+g(q,N)\right),\ \text{and}\\
\|H_{1n}\|_r^p=&d_{1n}^pm(D)\left(\frac{\sigma(N,p)}{D^p}+g(p,N)\right).
\end{array}
\]
Condition $(i)$, implies
\[
\begin{array}{rl}
\int_{\mathbb{R^N}}\alpha_1(x)F_1(H_{2n}(x))dx\geq&\int_{B(0,\frac{D}{2})}\alpha_1(x)F_1(d_{2n})dx=F_1(d_{2n})\|\alpha_1\|_{B(0,\frac{D}{2})},\ \text{and}\\
\int_{\mathbb{R^N}}\alpha_2(x)F_2(H_{1n}(x))dx\geq&\int_{B(0,\frac{D}{2})}\alpha_2(x)F_2(d_{1n})dx=F_2(d_{1n})\|\alpha_2\|_{B(0,\frac{D}{2})},
\end{array}
\]
for every $n \in N$. Then
\[
\begin{array}{rl}
I_\lambda(H_{1n},H_{2n})=&\Phi(H_{1n},H_{2n})-\lambda \Psi(H_{1n},H_{2n})\\
=&\frac{\|H_{1n} \|_r^p}{p}+\frac{\|H_{2n} \|_r^q}{q}-\lambda\int_{\mathbb{R}^N}\alpha_1(x)F_1(H_{2n}(x))dx
-\lambda\int_{\mathbb{R}^N}\alpha_2(x)F_2(H_{1n}(x))dx\\
\leq&\frac{d_{1n}^pm(D)\left(\frac{\sigma(N,p)}{D^P}+g(p,N)\right)}{p}
+\frac{d_{2n}^qm(D)\left(\frac{\sigma(N,q)}{D^q}+g(q,N)\right)}{q}\\
&\qquad \qquad-\lambda\left(
F_1(d_{2n})\|\alpha_1\|_{B(0,\frac{D}{2})}+F_2(d_{1n})\|\alpha_2\|_{B(0,\frac{D}{2})}.
\right)
\end{array}
\]
If $B<+\infty$ ($B_1,B_2<+\infty$), the conditions \eqref{B} implies that
\[
\begin{array}{l}
\text{there exists}\ N_1\ \text{such that for all}\ n\geq N_1\ \text{we have}\ F_1(d_{2n})>\varepsilon B_1 d_{2n}^p,\ \text{and}\\
\text{there exists}\ N_2\ \text{such that for all}\ n\geq N_2\ \text{we have}\ F_2(d_{1n})>\varepsilon B_2 d_{1n}^q.
\end{array}
\]
Then for every $n\geq N_{\varepsilon}:=\max\{N_1,N_2\}$,
\[
\begin{array}{rl}
I_\lambda(H_{1n},H_{2n})
\leq&\frac{d_{1n}^pm(D)\left(\frac{\sigma(N,p)}{D^P}+g(p,N)\right)}{p}
+\frac{d_{2n}^qm(D)\left(\frac{\sigma(N,q)}{D^q}+g(q,N)\right)}{q}\\
&\qquad \qquad-\lambda\varepsilon\left(
d_{1n}^pB_2\|\alpha_1\|_{B(0,\frac{D}{2})}+d_{2n}^qB_2\|\alpha_1\|_{B(0,\frac{D}{2})}.
\right)\\
=&d_{1n}^p\left(\frac{m(D)\left(\frac{\sigma(N,p)}{D^P}+g(p,N)\right)}{p}
-\lambda\varepsilon
B_1\|\alpha_2\|_{B(0,\frac{D}{2})}\right)\\
&\qquad \qquad +d_{2n}^q\left(\frac{m(D)\left(\frac{\sigma(N,q)}{D^q}+g(q,N)\right)}{q}-\lambda\varepsilon B_2\|\alpha_1\|_{B(0,\frac{D}{2})}\right).
\end{array}
\]
If we set $\Omega:=\max \left\{\frac{m(D)\left(\frac{\sigma(N,p)}{D^p}+g(p,N)\right)}{p\lambda B_1\|\alpha_2\|_{B(0,\frac{D}{2})}},\frac{m(D)\left(\frac{\sigma(N,q)}{D^q}+g(q,N)\right)}{q\lambda B_2\|\alpha_1\|_{B(0,\frac{D}{2})}}\right\}$, then for $
\varepsilon\in\left(\Omega, 1\right)$
one can get $\lim_{n\to +\infty} I_\lambda(H_{1n},H_{2n})=-\infty$.\\
If at least one of the $B_1$ or $B_2$ are $+\infty$. Let $B_1=+\infty$, and consider
$M_1>\Omega$, then by \eqref{B} there exists $N_{M_1}$ such that for every $n>N_{M_1}$, we have $F_1(d_{1n})>M_1d_{1n}^p$. Moreover, for every $n>N_{M_1}$
\[
\begin{array}{rl}
I_\lambda(H_{1n},H_{2n})
\leq&\frac{d_{1n}^pm(D)\left(\frac{\sigma(N,p)}{D^P}+g(p,N)\right)}{p}
+\frac{d_{2n}^qm(D)\left(\frac{\sigma(N,q)}{D^q}+g(q,N)\right)}{q}\\
&\qquad \qquad-\lambda\left(
d_{1n}^pM_1\|\alpha_2\|_{B(0,\frac{D}{2})}+d_{2n}^qM_1\|\alpha_1\|_{B(0,\frac{D}{2})}.
\right)\\
=&d_{1n}^p\left(\frac{m(D)\left(\frac{\sigma(N,p)}{D^P}+g(p,N)\right)}{p}
-\lambda M_1\|\alpha_2\|_{B(0,\frac{D}{2})}\right)\\
&\qquad \qquad +d_{2n}^q\left(\frac{m(D)\left(\frac{\sigma(N,q)}{D^q}+g(q,N)\right)}{q}-\lambda M_1\|\alpha_1\|_{B(0,\frac{D}{2})}\right).
\end{array}
\]
This implies that $\lim_{n\to +\infty} I_\lambda(H_{1n},H_{2n})=-\infty$.

Now, Theorem \ref{theo2.1} ($b$) implies, the functional $I_\lambda$ admits an unbounded sequence $\{u_n\}\subset X$ of critical points.
Considering Theorem \ref{palais}, these critical points are also critical points for the smooth and $O(N)$-invariant functional $I_\lambda:W^{1,p}(\mathcal{R}^N)\to \mathcal{R}$. Therefore, there is a sequence of radially symmetric weak solutions for the problem \eqref{eq1.1}, which are unbounded in $W^{1,p}(\mathcal{R}^N)$.
\end{proof}
Here we prove our second result which says that under different conditions the problem \eqref{eq1.1} has a sequence of
weak solutions, which converges weakly to zero.
\begin{theorem}\label{theo12}
Let $f_i:\mathbb{R}\rightarrow\mathbb{R}$ be two continuous and radially symmetric functions. Set
\[
\begin{array}{rl}
A':=&\underset{(\xi_1,\xi_2)\to 0^+}{\liminf}\frac{\underset{|t_1|\leq \xi_1}{\max}F_2(t_1)}{|\xi_1|^p}+
\frac{\underset{|t_2|\leq\xi_2}{\max}F_1(t_2)}{|\xi_2|^q},\\
B_1':=&\underset{\xi_2\to 0^+}{\limsup}\frac{F_1(\xi_2)}{|\xi_2|^p},\ \text{and}\ B_2':=\underset{\xi_1\to 0^+}{\limsup}\frac{F_2(\xi_1)}{|\xi_1|^q},
\end{array}
\]
where $B':=B_1'+B_2'$, $\xi=(\xi_1,\xi_2)$. If $\underset{(\xi_1,\xi_2)\geq 0}{\inf}F_2(\xi_1)+F_1(\xi_2)=0$ and $A'<\Omega m_0 B'$,
where $\Omega$ is given by \eqref{k},
for every $\lambda\in \Lambda':=\left]\Omega,\frac{1}{\left({pc_1^p}\|\alpha_2\|_1
+{qc_2^q}\|\alpha_1\|_1\right)A'}\right[$,
there exists an unbounded sequence of radially symmetric weak solutions for \eqref{eq1.1} in $X$.
\end{theorem}
\begin{proof}
For fixed $\lambda \in \Lambda'$, we consider $\Phi$, $\Psi$ and $I_\lambda$ as in Section 2. Knowing that
$\Phi$ and $\Psi$ satisfy the regularity assumptions in Theorem \eqref{theo2.1}, we show that $\lambda <\frac{1}{\delta}$.
We know that $\inf_X\Phi =0$. Set $\delta:=\liminf_{r\to 0^+}\varphi(r)$.  A computation similar to the one in the Theorem \ref{theo1} implies $\delta<\infty$ and if $\lambda \in \Lambda'$ then $\lambda <\dfrac{1}{\delta}$. A compaction (similar in the Theorem \ref{theo1}) shows that
$I_\lambda(H_{1n}, H_{2n})<0$ for $n$ large enough and thus zero is not a local minimum of $I_\lambda$.  Therefore, there exists a sequence $\{u_n\}\subset X$ of critical points of $I_\lambda$ which converges weakly to zero in $X$ as  $\lim_{n\to+\infty}\Phi(u_n)=0$.\\
Again, considering Theorem \ref{palais}, these critical points are also critical points for the smooth and
$O(N)$-invariant functional $I_\lambda:W^{1,p}(\mathcal{R}^N)\to \mathcal{R}$. Therefore, there is a sequence
of radially symmetric weak solutions for the problem \eqref{eq1.1}, which converges weakly to
zero in $W^{1,p}(\mathcal{R}^N)$.
\end{proof}

\section*{Acknowledgements}
The first author is partially supported by I.N.D.A.M - G.N.A.M.P.A. 2019 and the ``RUDN University Program 5-100''.

\end{document}